\documentclass[11pt,a4paper]{article}
\usepackage{mathrsfs}
\usepackage{amsmath}
\usepackage{bbm}
\usepackage{amsmath,amsthm,amssymb,amscd}
\usepackage{latexsym}
\usepackage{hyperref}
\usepackage[numbers,sort&compress]{natbib}
\usepackage{hypernat}
\usepackage{geometry}
\usepackage{enumerate}
\usepackage[utf8]{inputenc}
\usepackage[english]{babel}
\usepackage{authblk}

\newcommand{\normmm}[1]{{\left\vert\kern-0.25ex\left\vert\kern-0.25ex\left\vert #1
    \right\vert\kern-0.25ex\right\vert\kern-0.25ex\right\vert}}

\newtheorem{theorem}{Theorem}[section]

\newtheorem{corollary}[theorem]{Corollary}
\newtheorem{lemma}[theorem]{Lemma}

\theoremstyle{definition} \theoremstyle{remark}
\numberwithin{equation}{section}

\begin{document}
 \title{Norm Inequalities Related to Heinz and Logarithmic Means}
 \author{ Guanghua Shi\\
 School of Mathematical Sciences, Yangzhou University, Yangzhou, Jiangsu, China
 \\sghkanting@163.com}

\maketitle

\begin{abstract}
  In this paper, we got some refinements of the norm inequalities related to the Heinz mean and logarithmic mean.
\end{abstract}



\section{Introduction}

There are several means that interpolate between the geometric and arithmetic means. For instance, the Heinz mean $H_t(a,b),$  defined  by
\[
H_t(a,b)=\frac{a^{1-t}b^t+a^tb^{1-t}}{2} \quad \mbox{for} \quad 0\le t\le 1.
\]
In 1993, Bhatia-Davis \cite{BD93} obtained that if $A,B$ and $X$ are $n\times n$ matrices with $A,B$ positive semidefinite, then for every unitarily invariant norm $\normmm{\cdot},$
\begin{eqnarray}
\normmm{A^{\frac{1}{2}}XB^{\frac{1}{2}}}\le \frac{1}{2}\normmm{A^{1-t}XB^t+A^tXB^{1-t}}\le \frac{1}{2}\normmm{AX+XB}.
\end{eqnarray}

The logarithmic mean $L(a,b),$ defined by
\[
L(a,b)=\frac{a-b}{\log a- \log b}=\int_0^1a^tb^{1-t}dt,
\]
also interpolates the geometric and arithmetic means.
In 1999, Hiai-Kosaki \cite{HK99} proved the following inequality
\begin{eqnarray}
\normmm{A^{\frac{1}{2}}XB^{\frac{1}{2}}}\le \normmm{\int_0^1A^vXB^{1-v}dv}\le \frac{1}{2}\normmm{AX+XB}.
\end{eqnarray}
Moreover, in 2006, Drissi \cite{Dri06} proved that the following Heinz-logarithmic inequality
\begin{eqnarray}
\normmm{A^{1-t}XB^t+A^tXB^{1-t}}\le 2\normmm{\int_0^1A^vXB^{1-v}dv}
\end{eqnarray}
holds for $\frac{1}{4}\le t\le \frac{3}{4}.$

A complex-valued function $\varphi$ on $\mathbb{R}$ is said to be positive definite if the matrix $[\varphi(x_i-x_j)]$ is positive semidefinite for all choices of real numbers $x_1, x_2,\ldots, x_n,$ and $n=1,2,\ldots.$
Set $M(a,b)$ and $N(a,b)$ are two symmetric homogeneous means on $(0,\infty)\times (0,\infty).$ $M$ is said to strongly dominate $N,$ denoted by $M<< N,$ if and only if the matrix
\[
\left[\frac{M(\lambda_i,\lambda_j)}{N(\lambda_i,\lambda_j)}\right]_{i,j=1,\ldots,n}
\]
is positive semidefinite for any size $n$ and $\lambda_1,\ldots,\lambda_n>0.$
Drissi \cite{Dri06} proved that for $a,b\ge 0,$ $H_t(a,b)<<L(a,b)$ if and only if $\frac{1}{4}\le t\le \frac{3}{4}.$ In general, the inequality $M<<N$ could be stronger than the L\"{o}wner's order inequality $M\le N.$

For more operator or norm inequalities related to the Heinz mean and logarithmic mean we refer the readers to \cite{Ko98, Zhan99, KS14} and the references therein.

\section{Main results}
\begin{lemma} For $\sinh x$ and $\cosh x$, we have
\begin{enumerate}
\item[(i)]  If $|\beta|>|\alpha|>0,$ then the function $\frac{\cosh \alpha x}{\cosh \beta x}$
is positive definite.
\item[(ii)] If $|\beta|>|\alpha|>0$ with $\alpha, \beta$ the same sign, then $\frac{\sinh \alpha x}{\sinh \beta x}$
is positive definite.
\item[(iii)] If $\beta>0$ and $|\alpha|< \beta/2,$ then $\frac{\beta x \cosh \alpha x}{\sinh \beta x}$
is positive definite.
\end{enumerate}

\end{lemma}

\begin{proof}
We follow a similar argument as in Chapter 5 of \cite{Bh07}. From the product representations in p. 147-148 of \cite{Bh07},
\begin{eqnarray}
\frac{sinh x}{x}=\prod_{k=1}^{\infty}\left(1+\frac{x^2}{k^2\pi^2}\right), \qquad
\cosh x=\prod_{k=0}^{\infty}\left(1+\frac{4x^2}{(2k+1)^2\pi^2}\right),
\end{eqnarray}
we have
\begin{eqnarray}
\frac{\sinh \alpha x}{\sinh \beta x}=\frac{\alpha}{\beta} \prod_{k=1}^{\infty}\frac{1+\alpha^2x^2/k^2\pi^2}{1+\beta^2x^2/k^2\pi^2}, \qquad
\frac{\cosh \alpha x}{\cosh \beta x}=\prod_{k=0}^{\infty}\frac{1+4\alpha^2x^2/(2k+1)^2\pi^2}{1+4\beta^2x^2/(2k+1)^2\pi^2}.
\end{eqnarray}
Each factor in the product is of the form
\[
\frac{1+a^2x^2}{1+b^2x^2}=\frac{a^2}{b^2}+\frac{1-a^2/b^2}{1+b^2x^2}, \quad b^2>a^2.
\]
Since $1/(1+b^2x^2)$ is positive definite \cite[5.2.7]{Bh07}, it follows that the functions $\cosh \alpha x/\cosh \beta x$ is positive definite for $|\beta|>|\alpha|>0,$ and $\sinh \alpha x/\sinh \beta x$ is positive definite for $|\beta|>|\alpha|>0$ with $\alpha, \beta$ the same sign.

Since
\[
\frac{\beta x \cosh \alpha x}{\sinh \beta x}=\frac{\frac{\beta}{2} x }{\sinh \frac{\beta}{2}x}\cdot\frac{\cosh \alpha x}{\cosh \frac{\beta}{2}x}
\]
and $x/\sinh x$ is positive definite \cite[5.2.9]{Bh07}, it follows from the above argument that when $\frac{\beta}{2}>|\alpha|$ the function $\frac{\beta x \cosh \alpha x}{\sinh \beta x}$ is positive definite.
\end{proof}

Now, we define
\[
L_s(a,b)=\frac{a^{1-s}b^s-a^sb^{1-s}}{(1-2s)(\log a-\log b)}=\frac{1}{1-2s}\int_s^{1-s}a^vb^{1-v}dv
\]
for $a,b>0,$ and $0\le s<\frac{1}{2}.$ When $s=0,$ it is the logarithmic mean. So we can call it the generalized logarithmic mean. And we also have $\lim_{s\rightarrow {\frac{1}{2}}}L_s(a,b)=a^{\frac{1}{2}}b^{\frac{1}{2}}.$

\begin{theorem} For Heinz mean and the generalized logarithmic mean, we have
\begin{enumerate}
\item[(i)] If $0\le s<1/2,$ and $|1-2t|<\frac{1-2s}{2},$ then $H_t(a,b)<< L_s(a,b).$

\item[(ii)] If $|1-2t|<|1-2s|,$ then $H_t(a,b)<< H_s(a,b).$

\item[(iii)] If $0\le s<t< 1/2,$ or $1\ge s>t> 1/2,$ then $L_t(a,b)<< L_s(a,b).$
\end{enumerate}
\end{theorem}

\begin{proof}
By definition, $H_t(a,b)<< L_s(a,b)$ if
\[
[y_{i,j}]=\left[\frac{H_t(\lambda_i,\lambda_j)}{L_s(\lambda_i,\lambda_j)}\right]_{i,j=1,\ldots,n}
\]
is positive semidefinite. Set $\lambda_i=e^{x_i}$ and $\lambda_j=e^{x_j},$ with $x_i, x_j\in \mathbb{R}.$ Then
\[
y_{i,j}=(1-2s)\frac{(\frac{x_i-x_j}{2})(e^{(1-2t)\frac{x_i-x_j}{2}}+e^{(1-2t)\frac{x_j-x_i}{2}})}
{e^{(1-2s)\frac{x_i-x_j}{2}}-e^{(1-2s)\frac{x_j-x_i}{2}}}
\]
Thus the matrix $[y_{i,j}]$ is congruent to one with entries
\[
\frac{\beta(\frac{x_i-x_j}{2})\cosh (\alpha(\frac{x_i-x_j}{2}))}{\sinh(\beta (\frac{x_i-x_j}{2}))},
\]
where $\alpha=1-2t, \beta=1-2s.$ Hence $[y_{i,j}]$ is positive semidefinite if and only if the function $\frac{\beta x \cosh \alpha x}{\sinh \beta x}$ is positive definite, which by lemma 2.1 is correct.

Similarly, we have $H_t(a,b)<< H_s(a,b)$ if $\frac{\cosh \alpha x}{\cosh \beta x}$ is positive definite, and $L_t(a,b)<< L_s(a,b)$ if $\frac{\sinh \alpha x}{\sinh \beta x}$ is positive definite.
\end{proof}

\begin{theorem}
Let $A, B$ be any positive matrices. Then for any matrix $X$ and for $0\le s<1/2$ and $|1-2t|<(1-2s)/2,$ we have
\begin{eqnarray}
\normmm{A^{1-t}XB^t+A^tXB^{1-t}}\le \frac{2}{1-2s}\normmm{\int_s^{1-s}A^vXB^{1-v}dv}.
\end{eqnarray}
\end{theorem}

\begin{proof}
Firstly, we assume $A=B.$ Since $\normmm{\cdot}$ is unitarily invariant, we may suppose $A$ is diagonal with entries $\lambda_1,\ldots, \lambda_n.$ Then we have
\[
A^{1-t}XA^t+A^tXA^{1-t}=Y\circ (\int_s^{1-s}A^vXA^{1-v}dv),
\]
where $Y$ is the matrix with entries
\[
y_{i,j}=\frac{\lambda_i^t\lambda_j^{1-t}+\lambda_i^{1-t}\lambda_j^{t}}{\frac{\lambda_i^{1-s}\lambda_j^s-\lambda_i^s\lambda_j^{1-s}}{\log \lambda_i -\log \lambda_j}}
=\frac{2 H_t(\lambda_i,\lambda_j)}{(1-2s)L_s(\lambda_i,\lambda_j)}
\]
A well-known result related to the Schur multiplier norm \cite[Theorem 5.5.18, 5.5.19]{HJ91} says that if $Y$ is any positive semidefinite matrix, then for all matrix $X,$
\begin{eqnarray}
\normmm{Y\circ X}\le \max_{i}{y_{ii}}\normmm{X}\end{eqnarray}
for every unitarily invariant norm.
By Theorem 2.2, $Y$ is positive semidefinite. Applying $(2.4),$ we have
\begin{eqnarray}
\normmm{A^{1-t}XA^t+A^tXA^{1-t}}\le \frac{2}{1-2s}\normmm{\int_s^{1-s}A^vXA^{1-v}dv}.
\end{eqnarray}
Now replacing $A$ and $X$ in the inequality $(2.5)$ by the $2$ by $2$ matrices
$\left(\begin{array}{cc}
    A  & 0\\
    0 & B\\
  \end{array}\right)$
and
$\left(\begin{array}{cc}
    0 & X\\
    0 & 0\\
  \end{array}\right).$
  This gives the desired inequality $(2.3).$
\end{proof}

When $s=0,$ we get Drissi's result $(1.3).$ Moreover, when $s=0, t=1/2,$ we get the first inequality of $(1.2).$

\begin{theorem}
Let $A, B$ be any positive matrices. Then for any matrix $X$ and for $0\le s<1/2,$  we have
\begin{eqnarray}
\normmm{\int_s^{1-s}A^vXB^{1-v}dv}\le \frac{1-2s}{2}\normmm{A^{1-s}XB^s+A^sXB^{1-s}}.
\end{eqnarray}
\end{theorem}

\begin{proof}
Suppose $A$ is diagonal with entries $\lambda_1,\ldots, \lambda_n.$ Then we have
\[
\int_s^{1-s}A^vXA^{1-v}dv=Y\circ (A^{1-s}XA^s+A^sXA^{1-s}),
\]
where $Y$ is the matrix with entries
\[
y_{i,j}=\frac{\frac{\lambda_i^{1-s}\lambda_j^s-\lambda_i^s\lambda_j^{1-s}}{\log \lambda_i -\log \lambda_j}} {\lambda_i^s\lambda_j^{1-s}+\lambda_i^{1-s}\lambda_j^{s}}
=\frac{(1-2s)L_s(\lambda_i,\lambda_j)}{2H_s(\lambda_i,\lambda_j)}.
\]
By a similar argument as in the proof of Theorem 2.2, we know that $Y$ is positive definite if and only if
\[
\frac{\beta}{2}\frac{\sinh \beta x}{\beta x \cosh \beta x}=\frac{\beta}{2}\frac{\tanh \beta x}{\beta x}
\]
is positive definite for $\beta=1-2s>0,$ because $\tanh x/x$ is positive definite (see Bhatia \cite[5.2.11]{Bh07}).
Thus Applying $(2.4)$ we have
\begin{eqnarray}
\normmm{\int_s^{1-s}A^vXA^{1-v}dv}\le \frac{1-2s}{2}\normmm{A^{1-s}XA^s+A^sXA^{1-s}}.
\end{eqnarray}
Hence the desired result follows.
\end{proof}

When $s=0,$ we get the second inequality of $(1.2).$

\begin{theorem}
Let $A, B$ be any positive matrices. Then for any matrix $X$ and for $|1-2t|<|1-2s|$ with $1-2t, 1-2s$ the same sign, we have
\begin{eqnarray}
\normmm{A^{1-t}XB^t-A^tXB^{1-t}}\le \left|\frac{1-2t}{1-2s}\right|\normmm{A^{1-s}XB^s-A^sXB^{1-s}}.
\end{eqnarray}
\end{theorem}

\begin{proof}
Suppose $A$ is diagonal with entries $\lambda_1,\ldots, \lambda_n.$ Then we have
\[
A^{1-t}XA^t-A^tXA^{1-t}=Y\circ (A^{1-s}XA^s-A^sXA^{1-s}),
\]
where $Y$ is the matrix with entries
\[
y_{i,j}=\frac{\lambda_i^{1-t}\lambda_j^{t}-\lambda_i^t\lambda_j^{1-t}} {\lambda_i^{1-s}\lambda_j^{s}-\lambda_i^s\lambda_j^{1-s}}
\]
Put $\lambda_i=e^{x_i}$ and $\lambda_j=e^{x_j},$ with $x_i, x_j\in \mathbb{R}.$ Then
$Y$ is congruent to the matrix with entries
\[
\frac{\sinh (\alpha\frac{x_i-x_j}{2})}{\sinh (\beta\frac{x_i-x_j}{2})}
\]
where $\alpha=1-2t, \beta=1-2s.$ Since $(\sinh \alpha x)/(\sinh \beta x)$ is positive for $|\beta|>|\alpha|>0$ with $\alpha, \beta$ the same sign,  it follows that
$Y$ is positive definite. Applying the inequality $(2.4),$ we have
\begin{eqnarray}
\normmm{A^{1-t}XA^t-A^tXA^{1-t}}\le \left|\frac{1-2t}{1-2s}\right|\normmm{A^{1-s}XA^s-A^sXA^{1-s}}.
\end{eqnarray}
Hence the desired result follows.
\end{proof}

Set $s=0$ and $s=1,$ and combine the conclusions, we have the following inequality proved by Bhatia-Davis \cite{BD95},

\begin{corollary}
Let $A, B$ be any positive matrices. Then for any matrix $X$ and for $0\le t\le 1$  we have
\begin{eqnarray}
\normmm{A^{1-t}XB^t-A^tXB^{1-t}}\le |1-2t|\normmm{AX-XB}.
\end{eqnarray}
\end{corollary}

\subsection*{Acknowledgments}
The author acknowledges support from National Natural Science Foundation of China, Grant No: 12001477, and the Natural Science Foundation of Jiangsu Province for Youth, Grant No: BK20190874.

\end{document}